\documentclass[twoside,a4paper]{article}

\usepackage{amsmath,graphicx,amssymb,fancyhdr,amsthm,enumerate,stix,lie-hasse}
\newtheorem{thm}{Theorem}[section]

\theoremstyle{definition}

\theoremstyle{remark}  
  
\def\beq{\begin{eqnarray}}  
\def\eeq{\end{eqnarray}}  
\def\bsp{\begin{split}}  
\def\esp{\end{split}}

\def\d{\mathrm{d}}  
\def\span{\mathrm{span}}

\def \g{\mathfrak g}

\newcommand{\mf}[1]{{\mathfrak #1}}

\def\R{\mathbb{R}}
\def\C{\mathbb{C}}
\def\N{\mathcal{N}}

\def\Ric{\mathrm{Ric}}
\def\doplus{~\dot\oplus~}
\newcommand{\dotbigoplus}[2]{\overset{#2}{\underset{#1}{\dot\bigoplus}} ~}

\begin{document}   
   
\title{\Large\textbf{Left-invariant Pseudo-Riemannian metrics on Lie groups: \\The null cone II }}  
\author{{\large{Sigbj\o rn Hervik$^\clubsuit$}    }
 \vspace{0.3cm} \\
 $^\clubsuit$ Department of Mathematics and Physics, 
 University of Stavanger, \\  N-4036 Stavanger, Norway\\
{ \texttt{sigbjorn.hervik@uis.no,} }}
\date{\today}
\maketitle
\pagestyle{fancy}
\fancyhead{} 
\fancyhead[EC]{Sigbjørn Hervik}
\fancyhead[EL,OR]{\thepage} \fancyhead[OC]{Left-invariant Pseudo-Riemannian metrics on Lie groups: The null cone II}
\fancyfoot{} 

\begin{abstract}
We continue to study left-invariant pseudo-Riemannian metrics on Lie groups being in the null cone of the $O(p,q)$-action  using the moving bracket approach. In particular, the Lie algebra being in the null cone implies that the pseudo-Riemannian metric have all vanishing scalar curvature invariants (VSI). We consider \emph{all} Lie algebras of dimension $\leq 6$ and we find that all solvable Lie algebras, and non-trivially Levi-decomposable Lie algebras, of dimension $\leq 6$ are in the null cone, \emph{except} the 3-dimensional solvable  Lie algebra $\mf{s}_{3,3}$. For  $\mf{g}$  semi-simple, we also give a construction where $\mf{g}\oplus\R^m$ is in the null cone and give examples of such spaces for \emph{all}  the real simple Lie algebras $\mf{g}$. For example, for the exceptional split groups this construction places the split $\mf{e}_6\oplus\R^6$,  split $\mf{e}_7\oplus\R^7$ and  split $\mf{e}_8\oplus\R^8$ in the null cone of the $O(42,42)$, $O(70.70)$ and $O(128,128)$ action, respectively, and hence, their corresponding left-invariant pseudo-Riemannian metrics are VSI. 
\end{abstract} 
\section{Introduction}
In this paper we continue the work started in \cite{nullcone}  and investigate left-invariant pseudo-Riemannian metrics in the null cone and, in particular, we consider \emph{all} Lie algebras of dimension $\leq 6$.  We also do a systematic study of Lie algebras being direct sums of any semi-simple Lie algebra and $\R^m$. 

Previously, the study of left-invariant (pseudo-)Riemannian metrics have been focussed on Einstein metrics, or Ricci nil- or solsoliton metrics \cite{Wolter,Wolter2,Heber,L1,L2,L3,Sig,L4,Lauret06,Lauret09,Lauret10,LW11,L5,CF,solwaves,CR19a,CR19b,CR21}. 
The motivation for studying Lie algebras in the null-cone  comes from general relativity, pseudo-Riemannian geometry and invariant theory. All of these metrics in the null cone have the property that \emph{all polynomial scalar curvature invariants vanish}, they are so-called VSI spaces \cite{VSI,Higher,CFHP,Hervik12}.  This implies that if we compute any polynomial scalar curvature invariant it will be zero. The simplest scalars of this kind are the Ricci scalar, $R$, and the Kretchmann scalar being quadratic in the Riemann tensor. Polynomial curvature invariants can also be constructed from the covariant derivatives of the Riemann tensor: $\nabla{\rm Riem}$, $\nabla^2{\rm Riem}$, etc. Any scalar invariant constructed from these curvature tensors will vanish for the VSI-spaces, and in particular, the metrics considered here. The VSI spaces have shown to be important in alternative theories of gravity, and in higher gravity theories in particular \cite{HorowitzSteif,CGHP}. Requiring the Lie algebra being in the null cone also implies that the scalar invariants of the Killing operator will vanish implying the Killing operator is a nilpotent operator \cite{nullcone}. 

Mathematically speaking, the question at hand is to investigate which Lie algebras are in the null cone of the action of som $O(p,q)$ action. Some results were established in \cite{nullcone}: 
\begin{enumerate}
    \item If the Lie algebra $\mf{g}$ is semi-simple, then it is not in the null cone of any $O(p,q)$-action. 
    \item If the Lie algebra $\mf{g}$ is nilpotent, then it is in the null cone of some $O(p,q)$-action.
    \item If the Lie algebra $\mf{g}$ is completely solvable, then it is in the null cone of the split $O(p,q)$-action (i.e., either $O(p,p)$ or $O(p,p+1)$).
    \item If the Lie algebra $\mf{g}$ is in the null cone of the $O(p,p+k)$ action, then $\mf{g}$ has a nilpotent subalgebra of dimension $p+k$. 
\end{enumerate}
The aim of this paper is to systematically consider classes of Lie algebras and investigate whether they are in the null cone or not. 
We will  consider \emph{all} Lie algebras of dimension $\leq 6$. Of these, the solvable Lie are the most numerous, and we find they all are in the null cone, except for the 3-dimensional Lie algebra $\mf{s}_{3,3}$ \cite{SW}. This Lie algebra was pointed out in \cite{nullcone} not being in the null cone, and is not a completely solvable Lie algebra. Amazingly, this turns out to be the only solvable exception among solvable Lie algebras of dimensions $\leq 6$ for being in the null cone for some $O(p,q)$-action. 

We also find other examples of metrics not in the null cone, and all of these contain a semi-simple subalgebra. For example, in dimension 5, the only Lie algebras in the null cone having a simple subalgebra are the decomposable Lie algebra $\mf{s}\mf{l}(2,\R)\oplus\R^2$ and the nontrival Levi decomposable $\mf{s}\mf{l}(2,\R)\oplusrhrim\R^2$. 

 After this systematic treatment considering dimension by dimension, we follow a constructive approach to  studying left-invariant metrics on semi-simple Lie groups $G$, adding abelian dimensions: $G\times \mathbb{R}^m$, aiming to find the smallest number $m$ possible so that  $G\times \mathbb{R}^m$ is in the null cone. This is highly successful and we give a recipe for constructing null cone Lie algebras   for any real semi-simple group $G$. For example,  for all the split simple groups we show, for example, that the 256-dimensional Lie algebra $\mf{e}_8\oplus\R^8$ (direct sum) is in the null cone of the $O(128,128)$-action. In general, for the Split $G\times \mathbb{R}^m$, where $m$ equals the real rank of the simple group $G$, is in the null cone of the $O(p,p)$-action, where $p=\dim(G\times \R^m)/2$. 
 
 We also extend this construction to all real semi-simple groups $G$. For these the number $m$ are larger than the split case since the dimension of the largest compact subgroup increases and the nilpotent subgroup decreases. For the compact $G$  case, Markestad \cite{Markestad} showed that $m=\dim(G)$ would give a Lie algebra $\mf{g}\oplus\R^m$ in the null cone. 
 
 Our main results are: 
 \begin{thm}
The following classes of Lie algebras are in the null cone of some $O(p,q)$-action:
 \begin{itemize}
 \item{} All  solvable Lie algebras of dimension $\leq 6$, \emph{except} the 3-dimensional Lie algebra ${\mf s}_{3,3}$ given in \cite{SW}.
 \item{} All non-trivial Levi-decomposable Lie algebras of dimension $\leq 6$. 
 \item{} Lie algebras being a direct sum $\mf{g}\oplus\R^m$, where $\mf{g}$ is semi-simple, and $m\geq \dim(\mf{g}_0)$ where ${\mf g}=\dot\bigoplus_{\lambda=-\Delta}^{\Delta}{\mf g}_\lambda$ is a $Z$-grading of ${\mf g}$. 
 \end{itemize}  
 \end{thm} 
 As for notation: $\oplus$ is direct sum both as a vector space and as a Lie algebra; $\doplus$ is direct sum as a vector space; and $\oplusrhrim$ is semi-direct sum in terms of a non-trivial Levi decomposable Lie algebra. 
 
 The only 3-dimensional solvable Lie algebras \emph{not} being in the null cone is the following: 
 \beq
{\mathfrak s}_{3,3}:\quad [e_2,e_1]=\alpha e_2-e_3, \qquad  [e_3,e_1]=e_2+\alpha e_3, \qquad \alpha\geq 0. 
\eeq
and is not a completely solvable Lie algebra. There are other not-completely solvable Lie algebra in dimensions $\leq 6$, but these are actually in the null cone.  
 
 \subsection{The moving bracket approach}

Let us briefly recall the moving bracket approach. Let $V$ be a real vector space. Then a Lie algebra $\g\cong V$ (as a vector space) can be defined specifying a map $\mu: V\times V\rightarrow V$ using the skew-symmetric bracket $\mu$: $[X,Y]=\mu(X,Y)$, for all $X,Y\in V$ fulfilling the Jacobi identity: 
\beq
\mu(\mu(X,Y),Z)+\mu(\mu(Z,X),Y)+\mu(\mu(Y,Z),X)=0.
\eeq
As a tensor,  $\mu\in \mathcal{V}:=V\otimes\wedge^2V^*$ obeys the Jacobi identity. We equip $V$ with a (fixed)  pseudo-Riemannian metric $g$ of signature $(p,q)$ where $\dim(V)=p+q$ in the following way: Assume $p\leq q$ so that $q-p=k$,  we will choose a null basis $\{e_1,...,e_{2p+k}\}$ with a corresponding co-basis $\{\theta^1,...,\theta^{2p+k}\}$: 
\beq \label{metric}
g=2\theta^1\theta^2+...+2\theta^{2p-1}\theta^{2p}+\sum_{i=2p+1}^{2p+k}\left(\theta^i\right)^2.
\eeq
This choice gives a split of the tangent space into $T_pM=N_-\doplus H\doplus N_+$, where $N_-=\span\{e_1, e_3, ..., e_{2p-1}\}$ and  $N_+=\span\{e_2, e_4, ..., e_{2p}\}$ are null, and $H=\span\{e_{2p+1}, e_{2p+2},...,e_{2p+k}\}$ is the transverse space. 
 
The structure group of the frame bundle for the pseudo-Riemannian metric is the group $G=O(p,q)$.  The structure group acts naturally on the bracket: Given $A\in O(p,q)$, then
\[ (A\cdot\mu)(X,Y)=A^{-1}(\mu(AX,AY)), \]
Hence, for a given $\mu$ we can consider the orbit $G\mu\subset \mathcal{V}$. Note that there are two groups involved here: First, is the structure group $O(p,q)$ of the pseudo-Riemannian space; second is the Lie group, $G$, which we assume is the manifold itself. The latter has a Lie algebra represented by the bracket $\mu$. Here, we are interested in the action of the first group on the Lie algebra of the latter. 

Assuming a left-invariant metric the bracket $\mu$, we can be used to define a left-invariant torsion-free connection, $\nabla$, by requiring
\beq\label{torsionfree}
\nabla_{X}Y-\nabla_YX=\mu(X,Y), \quad \forall X,Y\in V.  
\eeq
The curvature tensors can from this be computed, as well as their scalar polynomial invariants. Polynomial curvature invariants are $O(p,q)-$invariant polynomials of the components of the curvature tensors, and are thus constants along orbits. They are polynomial in the curvature tensors and also polynomial in the structure constants. 

The left-invariant null-basis $\{e_a\}$ and its corresponding left-invariant co-basis $\{\theta^a\}$, can be used to describe the structure coefficients: $[e_a,e_b]=C^c_{~ab}e_c$. These represent the components of $\mu$ with respect to the null-frame.

We will here consider various Lie algebras and useful for this is the book \cite{SW} where all Lie algebras  of dimension $\leq 6$ are listed. Whenever we refer to a specific Lie algebra we will use the notation in \cite{SW} as we will here consider \emph{all} Lie algebras of dimension $\leq 6$.

\section{The null cone}
Under the orbit of $G=O(p,q)$ acting on $\mathcal{V}$ the null cone, $\N\subset\mathcal{V}$ is defined as \cite{KW}: 
\[ \N:=\left\{\text{Lie algebra }\mu\in \mathcal{V}: 0\in\overline{G\mu}\right\}; \]
i.e., the null cone consists of all orbits $G\mu$ having $0$ in their closure, $\overline{G\mu}$. 

A few results are worth repeating \cite{RS}: 
\begin{thm} \label{thmRS}
Given the null cone as defined above. Then: 
\begin{enumerate}
\item{} The orbit $\{0\}$ is the unique closed orbit in $\N$. 
\item{} Consider all polynomial invariants constructed from the bracket under the action of $O(p,q)$. Then given a Lie algebra $\mu\in\mathcal{V}$, these are all zero if and only if $\mu\in\N$. 
\item{} Given a $\mu\neq 0$, and a Cartan decomposition $\g ={\mathfrak k}\doplus{\mathfrak p}$ of $\mathfrak{o}(p,q)$, then there exists an $X\in \mathfrak{p}$ so that $\lim_{t\rightarrow \infty} e^{tX}\cdot\mu=0.$
\end{enumerate}
\end{thm}
 These results imply that 
{\it all $O(p,q)$-invariant polynomials of $\mu$ must vanish on $\N$ }, including curvature invariants which are necessarily all  zero as well, i.e., they are VSI spaces \cite{Hervik12}. 

Let us also mention the Killing form, which is a homogeneous polynomial of order 2 in $\mu$. The Killing operator ${K}\in{\rm End}(V)$ is defined by:
\[ B(v,w)=g({K}(v),w), \quad \forall v,w\in V, \]
where $B$ is the Killing form. The polynomial invariants of the Killing form, as well as the eigenvalues, need to be zero if $\mu$ is in the null cone. 

Note that the orbits
\[ GL(n,\R)\cdot\mu=\{A\cdot\mu \in\mathcal{V}~:~ A\in GL(n,\R)\} \] 
defines isomorphic Lie algebras over $\R$, and 
\[ O(p,q)\cdot\mu=\{A\cdot\mu  \in\mathcal{V}~:~ A\in O(p,q)\} \] 
defines isometric pseudo-Riemannian structures. Here, we will be interested in for which Lie algebras $\mu$ 
\[ (GL(n,\R)\cdot\mu)\cap\mathcal{N}\neq \emptyset. \] 
Note that $(GL(n,\R)\cdot\mu)\cap\mathcal{N}$ may have several disjoint components  representing non-isometric pseudo-Riemannian metrics on the Lie algebra $\mu$. 

\subsection{The boost weight decomposition and the null cone} 
Given a Lie algebra in the null cone, then there exists an $X\in \mf{p}$ so that $\lim_{t\rightarrow\infty}e^{tX}\cdot\mu = 0$. Since $X\in\mf{p}$ we can by an inner automorphism choose a  frame so that $X$ is in the Cartan algebra of $\mf{o}(p,q)$, i.e., $X\in\mf{a}$. Using the set of positive weights of $\mf{a}$, $\sf{b}_i$, $i=1, ... ,p$, assuming $p\leq q$,  we can decompose $\mu$ into simultaneous eigenvectors of the weights ${\sf b}_i$: 
\[ \mu =\sum_{{\sf b}\in \Delta} \mu_{\sf b}, \qquad {\sf b}=(b_1,b_2,..., b_p).\]
We will refer to the tuple $ {\sf b}=(b_1,b_2,..., b_p)$ as the \emph{boost-weight} of the components \cite{Hervik10}. 
Since the ${\sf b}_i$'s are also a basis of ${\mf{a}}$, then $X=\sum_{i=1}^px_i{\sf b}_i$. Consequently, 
\[ \lim_{t\rightarrow\infty}e^{tX}\cdot\mu_{\sf b}=\lim_{t\rightarrow\infty}e^{t(x_1b_1+x_2b_2+...+x_pb_p)}\mu_{\sf b}= 0\]  
if and only if:
\[ x_1b_1+x_2b_2+...+x_pb_p< 0.\] 
By an appropriate scaling, and using necessary discrete transformations, we can choose ${\sf b}\in \mathbb{Z}^p$, and the components $x_i$ are reversely well-ordered, $x_1\geq x_2\geq...\geq x_p\geq 0$, and: 
\beq \label{eq:bweq}
x_1b_1+x_2b_2+...+x_pb_p\leq -1.
\eeq
This is the fundamental equation we need to show the Lie algebra can satisfy: If for some choice of set $[x_1, x_2, x_3, ...,x_p]$, all components of $\mu$ satisfy eq.(\ref{eq:bweq}), then it is in the null cone.  This set $[x_1, x_2, x_3, ...,x_p]$ will be referred to as the \emph{class} of the Lie algebra in the null cone, and gives you the $X\in\mf{a}$ providing the limit $\lim_{t\rightarrow\infty}e^{tX}\cdot\mu_{\sf b}=0$. Note that the class of a  Lie algebra is not unique, however, here we will be mostly interested in the \emph{existence} of  a Lie algebra in the null cone and for this providing one class would be sufficient. 

\section{Lie algebras of dimension $\leq 3$}
The only 1-dimensional example is $\R$ which is trivially in the null cone since it is the origin itself. In dimension 2 and 3 we can only have Lorentzian signatures $(1,1)$ and $(1,2)$. Thus, the only possible classes are $[1]$ and $[\tfrac 12]$. This implies that the only allowable structure constants are: 
\[ C^2_{~12}, ~C^3_{~13},~C^2_{~13}.\]
Table \ref{tab:dim3} summarises these Lie algebras. The Lie algebras \emph{not} in the null cone are the solvable $\mf{s}_{3,3}$, and the simple algebras $\mf{s}\mf{u}(2)$ and $\mf{s}\mf{l}(2,\R)$. 
\begin{table}[ht]
    \centering
    \begin{tabular}{|c||c|c|}
    \hline 
 Lie algebra & null cone & Comment\\
 \hline\hline 
$\mf{s}_{2,1}$  &  $C^2_{~12}=1$ & CS \\
\hline\hline
$\mf{s}_{2,1}\oplus\R$  &  $C^2_{~12}=1$ & CS  \\
\hline \hline
$\mf{n}_{3,1}$  &  $C^2_{~13}=1$ & Nilpotent  \\
\hline 
$\mf{s}_{3,1}$  &  $C^2_{~12}=1,~C^3_{13}=a, \quad 0<|a|\leq 1$ & CS  \\
\hline
$\mf{s}_{3,2}$  &  $C^2_{~12}=C^2_{~13}=C^3_{~13}=1$ & CS  \\
\hline
$\mf{s}_{3,3}$  &  $\notin$ & Not CS  \\
\hline
$\mf{s}\mf{u}(2)$& $\notin$ & semi-simple \\
\hline 
$\mf{s}\mf{l}(2,\R)$& $\notin$ & semi-simple \\
\hline 
    \end{tabular}
    \caption{List of non-trivial Lie algebras of dimension $\leq 3$. Indicated is also whether they are in the null cone or not. Examples of structure coefficients are given if they are an element of the null cone.}
    \label{tab:dim3}
\end{table}

\section{Lie algebras of dimension 4}
In dimensions 4, we can have both the Lorentzian case, $O(1,3)$ and the neutral case $O(2,2)$.  Here, when we refer to $e_\mu$ we refer to the basis given in \cite{SW}. For a basis with respect to the pseudo-Riemannian metric, we will just give the number '1', or '2' etc. 

The neutral case allows for most Lie algebras in its null cone, and we will 
consider the solvable algebras $\mf{s}_{4,1-11}$ first. For all of these we can identify $e_4$ with '1' so that 
\[ C^\mu_{~1\nu}=S^\mu_{~\nu},\]
which is a matrix representing the adjoint map with respect to vector '1'. 
For the matrix $S^\mu_{~\nu}$ the problematic entry is $S^3_{~2}$ which has boost weight $(0,1)$ in the neutral case. By choosing $e_{1}=$'2', then $S^3_{~2}=0$ and it is in the null cone for algebras $\mf{s}_{4,1-5}$.  For  $\mf{s}_{4,6-11}$ they have non-trivial nilradical which can be represented with $C^2_{~34}=1$. This has boost-weight $(-1,0)$ and would thus place these algebras in the null cone as well. For the not completely solvable Lie algebra $\mf{s}_{4,12}$ a specific choice can be made which places this in the null cone too. This choice is given in Table \ref{tab:dim4} which also summarises the 4-dimensional  Lie algebras. 

To be specific, the Lie algebra $\mf{s}_{4,12}$ can be realised by the left-invariant one-forms: 
\[ \theta^1=dw, ~~ \theta^2=dx +ydz + xdw, ~~\theta^3=dz, ~~ \theta^4=dy - xdz + ydw\] 
This is of class $[1,1]$ a  metric in the null cone is
\[ g=2\theta^1\theta^2+2\theta^3\theta^4,\]  
\and is in the null cone of the $O(2,2)$-action. 

The only 4-dimensional Lie algebra \emph{not} in the null cone of the $O(1,3)$ or $O(2,2)$-action is $\mf{s}\mf{u}(2)\oplus \R$.

\begin{table}[ht]
    \centering
    \begin{tabular}{|c||c|c|}
    \hline 
 Lie algebra & null cone & Comment\\
 \hline\hline 
$\mf{s}_{2,1}\oplus \R^2$  &  $C^2_{~12}=1$ & N/L, CS \\
\hline
$\mf{s}_{2,1}\oplus\mf{s}_{2,1}$  &  $C^2_{~12}=C^4_{~34}=1$ &N, CS  \\
\hline \hline
$\mf{n}_{3,1}\oplus \R$  &  $C^2_{~13}=1$ & N/L, Nilpotent  \\
\hline 
$\mf{s}_{3,1}\oplus \R$  &  $C^2_{~12}=1,~C^3_{13}=a, \quad 0<|a|\leq 1$ & N/L, CS  \\
\hline
$\mf{s}_{3,2}\oplus \R$  &  $C^2_{~12}=C^2_{~13}=C^3_{~13}=1$ & N/L, CS  \\
\hline
$\mf{s}_{3,3}\oplus \R$  &  $C^3_{~13}=C^4_{14}=\alpha, ~C^3_{14}=-C^4_{13}=1, \quad 0\leq\alpha$ &N, Not CS  \\
\hline
$\mf{s}\mf{u}(2)\oplus \R$& $\notin$ & \\
\hline 
$\mf{s}\mf{l}(2,\R)\oplus \R$& $C^1_{~13} =C^4_{~34} = 2,~ C^3_{14} = 1.$ & N\\
\hline \hline
$\mf{n}_{4,1}$  &  $C^2_{~13}=C^4_{~12}=1$ &  Nilpotent  \\
\hline 
$\mf{s}_{4,1-5}$  & $C^\mu_{~1\nu}=S^\mu_{~\nu}, ~S^3_{~2}=0$ & Nilradical $\R^3$\\
\hline
$\mf{s}_{4,6-11}$  & $C^\mu_{~1\nu}=S^\mu_{~\nu}, ~S^3_{~2}=0, ~C^2_{~34}=1$ &  Nilradical $\mf{n}_{3,1}$\\
\hline
$\mf{s}_{4,12}$  & $C^2_{~12}=C^4_{~14}=C^2_{~34}=C^4_{~23}=1$ & Nilradical $\R^2$ \\
\hline
    \end{tabular}
    \caption{List of non-trivial Lie algebras of dimension $4$. Indicated is also whether they are in the null cone or not. Examples of structure coefficients are given if they are an element of the null cone. N=neutral signature. L=Lorentzian signature.}
    \label{tab:dim4}
\end{table}

\section{Lie algebras of dimension 5}
For the decomposable Lie algebras, many can be read off the lower dimensional ones. For example, if $\mf{g}_4$ is a four-dimensional Lie algebra in the null cone, then the five-dimensional $\mf{g}_4\oplus\R$ would be in it too. Moreover, if $\mf{g}_3$ is a three-dimensional Lie algebra in the null cone, then $\mf{g}_3\oplus{\mf s}_{2,1}$ can be achieved by simply by appending the last two dimensions as a pair of non-orthogonal null directions. 
In Table \ref{tab:dim5} all the decomposable are given and some of their structure constants are given explicitly as examples in the null cone. 

The Lie algebra $\mf{s}\mf{l}(2,\R)\oplusrhrim \R^2$ requires some investigation. This is the only 5-dimensional Lie algebra with a non-trivial Levi decomposition. From \cite{nullcone} we know that $N^+\doplus H=\span\{2,4,5\}$ has to be nilpotent. Furthermore, we should also have that $C^2_{~54}$ is the only amongst these which should be non-zero. Indeed, this can be accomplished if we choose $\mf{s}\mf{l}(2,\R)=\span\{1,3,5\}$ and $\R^2=\span\{2,4\}$ with the following non-zero structure constants: 
\[ \mf{s}\mf{l}(2,\R)\oplusrhrim \R^2:~C^1_{~13}=C^5_{~35}=2, \quad C^3_{~15}=1, \qquad C^2_{~32}=C^4_{~43}=C^2_{~54}=C^4_{~21}=1. \] 
Thus, this is in the null cone and of class $[2,1]$. Interestingly, if you compute the Riemann tensor, it has only one independent component, and thus the Ricci tensor for this metric is remarkably simple: $\Ric=-6\theta^3\theta^3$,  
while the Killing operator is of the form $B={\bf 2}\oplus{\bf 3}$.

Consider the non-decomposable solvable algebras $\mf{s}_{5,\bullet}$ in the $O(2,3)$ case. For all of these we can identify $e_5$ with '1' so that 
\[ C^\mu_{~1\nu}=S^\mu_{~\nu}.\]
For the matrix $S^\mu_{~\nu}$ the problematic entries are $S^3_{~2}$ and $S^5_{~2}$. By choosing $e_{1}=$'2', then $S^3_{~2}=S^5_{~2}=0$ for all, except possibly for $\mf{s}_{5,13}$. This itself implies it is in the null cone for algebras $\mf{s}_{5,1-12}$. For $\mf{s}_{5,13}$ we can, in addition, choose $e_2=$'3' (or '4'). Then $\mf{s}_{5,13}$ is also in the null cone. 

For the algebras $\mf{s}_{5,14-32}$, they have an $\mf{n}_{3,1}$ nilradical. Choose, $C^2_{34}=1$ representing the nilradical, then $\mf{s}_{5,14-32}$ is in the null cone. 

For the algebras $\mf{s}_{5,33-43}$, they have an $\mf{n}_{4,1}$ nilradical. If we choose, $C^2_{34}=C^4_{~35}=1$ representing the nilradical, then $\mf{s}_{5,14-32}$ is in the null cone too. 

Finally, the remaining Lie algebras $\mf{s}_{5,44-45}$ they also have an $\mf{n}_{3,1}$ nilradical. A wiser choice is to identify $e_4$ as '1', $e_5$ as '3' and $e_1$ as '2'. Then to identify $e_2$ with '4' and $e_3$ with '5' would put $\mf{s_{5,44}}$ in the null cone. On the other hand,   $\mf{s_{5,45}}$ seems more difficult and it is a not a completely solvable Lie algebra since the adjoint map ${\rm Ad}_1$ has complex eigenvalues. The Lie algebra $\mf{s_{5,45}}$ is the Lie algebra of homotheties of the 3-dimensional nilgeometry ${\sf Nil}^3$. Viewed as horospheres of the complex hyperbolic plane ${\mathbb{H}}_\C^2$, the algebra is a solvable subalgebra of the Lie algebra of the isometry group $SU(2,1)$. Indeed, we can find this in the null cone also, for example, by chosing the following left-invariant null-frame on $\mf{s_{5,45}}$:
\beq\label{s545}
\mf{s_{5,45}}: \quad 
\begin{cases}
\theta^1 = \d v \\
\theta^2=e^{2w}\left(\d x+\tfrac 12(y\d z-z\d y)\right)\\
\theta^3 = \d w\\ 
\theta^4= e^{w}\left(\cos v\d y+\sin v\d z\right) \\
\theta^5= e^{w}\left(-\sin v\d y+\cos v\d z\right)
\end{cases}
\eeq
The structure constants in this case are:
\[ \mf{s_{5,45}}: ~C^2_{~23}=2, ~~C^2_{~45}=C^4_{~15}=C^4_{~34}=C^5_{~35}=-1,~~ C^5_{~14}=1,\]
and is therefore in the null cone of class $[2,1]$. 

Table \ref{tab:dim5} summarises the 5-dimensional  Lie algebras, and in particular: \emph{All 5-dimensional solvable Lie algebras are in the null cone of the $O(1,4)$ or the $O(2,3)$-action. }
 
\begin{table}[ht]
    \centering
    \begin{tabular}{|c||c|c|}
    \hline 
 Lie algebra & null cone & Comment\\
 \hline\hline 
$\mf{s}_{2,1}\oplus \R^3$  &  $C^2_{~12}=1$ &  CS \\
\hline
$\mf{s}_{2,1}\oplus\mf{s}_{2,1}\oplus\R$  &  $C^2_{~12}=C^4_{~34}=1$ &CS  \\
\hline \hline
$\mf{n}_{3,1}\oplus \R^2$  &  $C^2_{~13}=1$ &  Nilpotent  \\
\hline 
$\mf{s}_{3,1-3}\oplus \R^2$  &  See table \ref{tab:dim4} &  \\
\hline
$\mf{s}\mf{u}(2)\oplus \R^2$& $\notin$ & \\
\hline 
$\mf{s}\mf{l}(2,\R)\oplus \R^2$& $C^1_{~13} =C^4_{~34} = 2,~ C^3_{14} = 1.$ & \\
\hline \hline
$\mf{n}_{3,1}\oplus \mf{s}_{2,1}$  &  $C^2_{~15}=1$, $C^{4}_{~34}=1$ &  Nilpotent$\oplus\mf{s}_{2,1}$  \\
\hline 
$\mf{s}_{3,1}\oplus \mf{s}_{2,1}$  &  $C^2_{~12}=1,~C^5_{15}=a,~C^4_{~34}=1, \quad 0<|a|\leq 1$ & CS  \\
\hline
$\mf{s}_{3,2}\oplus \mf{s}_{2,1}$  &  $C^2_{~12}=C^2_{~15}=C^5_{~15}=1, ~C^4_{~34}=1$ & CS  \\
\hline
$\mf{s}_{3,3}\oplus \mf{s}_{2,1}$  &  $C^2_{~12}=C^4_{14}=\alpha, ~C^4_{12}=-C^2_{14}=1, ~C^5_{~35}=1, ~\quad 0\leq\alpha$ & Not CS  \\
\hline
$\mf{s}\mf{u}(2)\oplus \mf{s}_{2,1}$& $\notin$ & \\
\hline 
$\mf{s}\mf{l}(2,\R)\oplus \mf{s}_{2,1}$& $\notin$ & \\
\hline \hline
$\mf{s}\mf{l}(2,\R)\oplusrhrim \R^2$& $C^1_{~13}=C^5_{~35}=2,~~C^3_{~15}=C^2_{~32}=C^4_{~43}=C^2_{~54}=C^4_{~21}=1.$ & LD\\
\hline \hline
$\mf{n}_{4,1}\oplus\R$  &  $C^2_{~13}=C^4_{~12}=1$ & Nilpotent  \\
\hline 
$\mf{s}_{4,1-12}\oplus\R$  & See table \ref{tab:dim4} & \\
\hline\hline 
$\mf{n}_{5,1-6}$  & $\in$ & Nilpotent \\
\hline
$\mf{s}_{5,1-13}$  &  $\in$, see text. & Nilradical $\R^4$\\
\hline
$\mf{s}_{5,14-32}$  &  $\in$, see text.  & Nilradical $\mf{n}_{3,1}\oplus\R$\\
\hline 
$\mf{s}_{5,33-38}$  & $\in$, see text. & Nilradical $\mf{n}_{4,1}$ \\
\hline
$\mf{s}_{5,39-43}$  & $\in$, see text. & Nilradical $\R^3$ \\
\hline
$\mf{s}_{5,44}$  & $\in$, see text. & Nilradical $\mf{n}_{3,1}$ \\
\hline
$\mf{s}_{5,45}$  & $\in$, see text. &  Nilradical $\mf{n}_{3,1}$\\
\hline
\hline
    \end{tabular}
    \caption{List of non-trivial Lie algebras of dimension $5$. Indicated is also whether they are in the null cone or not. Examples of structure coefficients are given if they are an element of the null cone. LD=non-trivial Levi decomposition. }
    \label{tab:dim5}
\end{table}

Let us consider the Lie algebras $\mf{s}\mf{l}(2,\R)\oplusrhrim \R^2$  and $\mf{s_{5,45}}$ in a bit more detail. We  have the metric 
\[ g=2\theta^1\theta^2+\theta^3\theta^4+(\theta^5)^2.\]
It may not look obvious that these are indeed VSI metrics. Let us therefore prove this using invariant-preserving diffeomorphisms as well. 
First, consider $\mf{s_{5,45}}$ where  it is convenient to introduce a different set of variables: 
\[ \begin{cases}
W=w \quad 
X=xe^{2w}, \quad  
V=v,\\
Y=e^w(y\cos v+z\sin v),\quad 
Z=e^w(-y\sin v+z\cos v).
\end{cases}\]
The metric can now be written
\beq g&=&2dV\left[dX-2XdW+\tfrac 12(Y^2+Z^2)dV+\tfrac 12(YdZ-ZdY)\right] \nonumber \\
&& +2dW(dY-YdW-ZdV)+(dZ-ZdW+YdV)^2. 
\eeq
Second, consider  the diffeomorphism (recall it is of class $[2,1]$): 
\[ \phi_\lambda: ~(V,X,W,Y,Z)\mapsto (Ve^{-2\lambda},Xe^{2\lambda},We^{-\lambda},Ye^{\lambda},Z)\]
so that 
\beq \phi_{\lambda}^*g=g_\lambda&=&2dV\left[dX-2XdWe^{-\lambda}+\tfrac 12(Y^2e^{-2\lambda}+Z^2e^{-4\lambda})dV+\tfrac{e^{-\lambda}}2(YdZ-ZdY)\right] \nonumber \\
&& +2dW(dY-e^{-\lambda}YdW-e^{-3\lambda}ZdV)+(dZ-e^{-\lambda}ZdW+e^{-\lambda}YdV)^2. 
\eeq
Since this is a homogeneous space all invariants will be constants, and hence, for any curvature invariant $I$: $\phi^*_{\lambda}I=I$. 
Then considering the limit 
\[ \lim_{\lambda\rightarrow\infty}\phi_{\lambda}^*g=\lim_{\lambda\rightarrow\infty}g_\lambda=2dVdX+2dWdY+dZ^2,\]
which is flat space. Consequently, the $\phi_\lambda$ which is an invariant-preserving diffeomorphism, proves that the pseudo-Riemannian space defined by the left-invariant frame (\ref{s545}) is indeed a VSI space.  

Now, consider $\mf{s}\mf{l}(2,\R)\oplusrhrim \R^2$ where we can find a left-invariant frame: 
\beq\label{sl2rr2}
\mf{s}\mf{l}(2,\R)\oplusrhrim \R^2: \quad 
\begin{cases}
\theta^1 = \d z-2z\d x \\
\theta^2=\d w+w\d x+y[(wz+v)\d x+z\d w-dv]\\
\theta^3 = \d x(1+2yz)-y\d z\\ 
\theta^4= \d v-z\d w-(v+wz)\d x \\
\theta^5= \d y+2y\d x(1+yz)-y^2\d z
\end{cases}
\eeq
The following diffeomorphism will now do the trick: 
\[  \phi_\lambda: ~(z,w,x,v,y)\mapsto (e^{-2\lambda}z,e^{2\lambda}w,e^{-\lambda}x,e^\lambda v,y),\]
so that
 \beq \lim_{\lambda\rightarrow\infty}\phi_{\lambda}^*g=\lim_{\lambda\rightarrow\infty}g_\lambda&=&2dzdw+2dxdv+(dy)^2. 
 \eeq
 Therefore, this metric on $\mf{s}\mf{l}(2,\R)\oplusrhrim \R^2$ is a VSI space. One can also see that this diffeomorphism brings the structure constants to $0$, and hence, $\mu$ is indeed in the null cone as well. Note also that these metrics explicitly confirm that these classes of metrics are in the more general class of $\mathcal{I}$-degenerate metrics studied in \cite{Hervik14}.

\section{Lie algebras of dimension 6}

For the direct sum decomposable Lie algebras, it is trivial that for all 5-dimensional Lie algebra $\mf{g}$ being in the null cone, then also $\mf{g}\oplus \R$ will be in the null cone. Hence, it is sufficient to consider those 6-dimensional Lie algebras not being of this type. 

\subsection{Solvable Lie algebras of dimension 6}
It is therefore sufficient to consider solvable (including nilpotent) Lie algebras of the following form: 
$\mf{s}_{2,1}\oplus\mf{s}_{2,1}\oplus\mf{s}_{2,1}$, $\mf{s}_{4,\bullet}\oplus\mf{s}_{2,1}$, $\mf{s}_{3,\bullet}\oplus\mf{s}_{3,\bullet}$ and $\mf{s}_{6,\bullet}$. 

\paragraph{Lie algebras $\mf{s}_{2,1}\oplus\mf{s}_{2,1}\oplus\mf{s}_{2,1}$, $\mf{s}_{4,\bullet}\oplus\mf{s}_{2,1}$:} For the first two categories, we can simply use the four-dimensional solvable Lie algebra (they are all in the null cone) and append a 2-dimensional $\mf{s}_{2,1}$: $C^6_{~56}=1$. Hence, all of the solvable (including nilpotent) Lie algebras of type $\mf{s}_{2,1}\oplus\mf{s}_{2,1}\oplus\mf{s}_{2,1}$, $\mf{s}_{4,\bullet}\oplus\mf{s}_{2,1}$ are in the null cone. 

\paragraph{Lie algebras $\mf{s}_{3,\bullet}\oplus\mf{s}_{3,\bullet}$:} For the nilpotent case, $\mf{g}=\mf{n}_{3,1}$, then $\dim([\mf{g},\mf{g}])=1$; and if $\mf{g}$ any of the other 3-dim solvable algebras then $\dim([\mf{g},\mf{g}])=2$. If one of the algebras is nilpotent, then choose $\span\{e_2\}=[\mf{n}_{3,1},\mf{n}_{3,1}]$, and $\span\{e_2,e_4,e_6\}=\mf{n}_{3,1}$. For the other algebra, choose $\span\{e_1\}=\mf{g}/[\mf{g},\mf{g}]$ if $\mf{g}$ is solvable, or $\span\{e_5\}=[\mf{n}_{3,1},\mf{n}_{3,1}]$ it is nilpotent. Then in the  class $[3,1,1]$, $\mf{s}_{3,\bullet}\oplus\mf{s}_{3,\bullet}$ is in the null cone. 

For $\mf{s}_{3,\bullet}\oplus\mf{s}_{3,\bullet}$ where both are of the solvable algebras $\mf{s}_{3,1-3}$. Then choose $\span\{e_1\}=\mf{s}_{3,\bullet}/[\mf{s}_{3,\bullet},\mf{s}_{3,\bullet}]$ and  $\span\{e_1,e_5,e_6\}=\mf{s}_{3,\bullet}$, for one, and $\span\{e_3\}=\mf{s}_{3,\bullet}/[\mf{s}_{3,\bullet},\mf{s}_{3,\bullet}]$ and  $\span\{e_3,e_4,e_2\}=\mf{s}_{3,\bullet}$ for the other. Then this would be in the null cone of class $[3,3,1]$. 

\paragraph{Lie algebras $\mf{s}_{6,\bullet}$:} For  nilpotent 6-dimensional indecomposable Lie algebras $\mf{n}_{6,1-22}$, they are all in the null cone since any nilpotent Lie algebra is \cite{nullcone}. 

The solvable Lie algebras $\mf{s}_{6,1-242}$ can be split into classes according to their nilradical.  Here, when we refer to $e_\mu$ we refer to the basis given in \cite{SW}. For a basis with respect to the pseudo-Riemannian metric, we will just give the number '1', or '2' etc. 
\begin{itemize}
\item{$\mf{s}_{6,1-21}$, nilradical $\R^5$:} Here, choose the $e_6$ to correspond to '1'. Then  the adjoint w.r.t. the vector '1' is ${\rm Ad}_{1}=(C^\mu_{~1\nu})=(S^\mu_{~\nu})$. If we set $e_1=$'2' then this would put all these algebras in the null cone except possibly $\mf{s}_{6,16}$ and $\mf{s}_{6,21}$ (both are not completely solvable). For  these  we should  set $e_5=$'2' instead. Then all of these would be in the null cone in class $[1,1,1]$. 
\item{$\mf{s}_{6,22-90}$, nilradical $\mf{n}_{3,1}\oplus\R^2$:}
Again define $e_6=$'1' and the adjoint w.r.t. the vector '1' is ${\rm Ad}_{1}=(C^\mu_{~1\nu})=(S^\mu_{~\nu})$. The nilradical can now be defined by $C^2_{~46}=1$. Hence, we need to consider class $[3,1,1]$. All solvable Lie algebras $\mf{s}_{6,22-90}$ now fall into the null cone. 
\item{$\mf{s}_{6,91-116}$, nilradical $\mf{n}_{4,1}\oplus\R$:} 
Again define $e_6=$'1' and the adjoint w.r.t. the vector '1' is ${\rm Ad}_{1}=(C^\mu_{~1\nu})=(S^\mu_{~\nu})$. Define the nilradical as $C^2_{~46}=C^4_{~56}=1$. In class $[3,1,1]$ the Lie algebras $\mf{s}_{6,91-116}$ will be in the null cone. 

\item{$\mf{s}_{6,117-149}$, nilradical $\mf{n}_{5,1}$:}
Again define $e_6=$'1' and the adjoint w.r.t. the vector '1' is ${\rm Ad}_{1}=(C^\mu_{~1\nu})=(S^\mu_{~\nu})$. Define the nilradical as $C^2_{~65}=C^4_{~35}=1$. Then all of these are in the class $[1,1,1]$ of the null cone. Note that for the Lie algebras $\mf{s}_{6,134}$, $\mf{s}_{6,145}$ and $\mf{s}_{6,147}$, are not completely solvable with '1' not being an eigenvalue of ${\rm Ad}_1$. However, by this choice the structure constants $C^4_{~12}$ and $C^2_{~14}$ are still allowable within the $[1,1,1]$ class. 

\item{$\mf{s}_{6,150-157}$, nilradical $\mf{n}_{5,2}$:}
Again define $e_6=$'1' and the adjoint w.r.t. the vector '1' is ${\rm Ad}_{1}=(C^\mu_{~1\nu})=(S^\mu_{~\nu})$. Define the nilradical as $C^2_{~63}=C^4_{~65}=C^6_{~35}=1$. Then all of these are in the class $[1,1,1]$ and in the null cone. 
\item{$\mf{s}_{6,158-182}$, nilradical $\mf{n}_{5,3}$:}
Again define $e_6=$'1' and the adjoint w.r.t. the vector '1' is ${\rm Ad}_{1}=(C^\mu_{~1\nu})=(S^\mu_{~\nu})$. Define the nilradical as $C^2_{~46}=C^2_{~35}=1$. Then this is in the class $[3,1,1]$ of the null cone. 
\item{$\mf{s}_{6,183-189}$, nilradical $\mf{n}_{5,4}$:}
Define $e_6=$'1' and the adjoint w.r.t. the vector '1' is ${\rm Ad}_{1}=(C^\mu_{~1\nu})=(S^\mu_{~\nu})$. Define the nilradical as $C^2_{~45}=C^2_{~36}=C^4_{~56}=1$. Then they are in the class $[3,1,1]$ in the null cone. 
\item{$\mf{s}_{6,190-196}$, nilradical $\mf{n}_{5,5}$:}
Define $e_6=$'1', the adjoint w.r.t. the vector '1' is ${\rm Ad}_{1}=(C^\mu_{~1\nu})=(S^\mu_{~\nu})$, and  define the nilradical as $C^2_{~34}=C^4_{~36}=C^6_{~35}=1$. Then this is of class $[1,1,1]$ and in the null cone. 

\item{$\mf{s}_{6,197}$, nilradical $\mf{n}_{5,6}$:}
Define $e_6=$'1', the adjoint w.r.t. the vector '1' is ${\rm Ad}_{1}=(C^\mu_{~1\nu})=(S^\mu_{~\nu})$, and  define the nilradical as $C^2_{~34}=C^4_{~36}=C^6_{~35}=C^2_{~56}=1$. Then this is also of class $[1,1,1]$ and in the null cone. 

\item{$\mf{s}_{6,198-228}$, nilradical $\R^4$:}
Here, let us identify $\{e_6,e_5,e_4,e_3, e_2, e_1\}=\{1, 3, 5, 6, 4, 2\}$. For the class $[3,3,1]$, this places all of these Lie algebras in the null cone. 

\item{$\mf{s}_{6,229-241}$, nilradical $\mf{n}_{3,1}\oplus\R$:}
Here, we do the same identification $\{e_6,e_5,e_4,e_3, e_2, e_1\}=\{1, 3, 5, 6, 4, 2\}$ so that the nilradical is represented with $C^2_{~46}=1$. Considering the class $[3,1,1]$ these Lie algebras are then all in the null cone. 
\item{$\mf{s}_{6,242}$, nilradical $\mf{n}_{4,1}$:} It was shown in \cite{nullcone} that this is in the null cone of class $[2,1]$ of the $O(2,4)$-action. 
\end{itemize}

To summarize: \emph{All 6-dimensional solvable Lie algebras are in the null cone of the $O(3,3)$-action.}

\subsection{Direct sum decomposable Lie algebras of dimension 6}
 We will not consider the Lie algebras of the form $\mf{g}\oplus \R$ where  the 5-dimensional Lie algebra $\mf{g}$ is in the null cone, since it trivially will be in the null cone by trivally appending a dimension. Henceforth, only the ones $\mf{g}\oplus \R$ where $\mf{g}$ is not in the null cone, will be considered. The list of remaining decomposable Lie algebras of dimension 6 are then: 
 \begin{enumerate}
 \item{$\mf{s}\mf{u}(2)\oplus \R^3$:} This is in the null cone, as shown by Markestad \cite{Markestad}. It is of class $[1,1,1]$. 
 \item{$\mf{s}\mf{u}(2)\oplus \mf{s}_{2,1}\oplus\R$:} Not in the null cone. 
 \item{$\mf{s}\mf{l}(2,\R)\oplus \mf{s}_{2,1}\oplus\R$:} In the null cone since we can consider the 4-dimensional neutral metric of $\mf{s}\mf{l}(2,\R)\oplus \R$, and append it with $\mf{s}_{2,1}$: $C^6_{~56}=1$. 
 
 \item{($\mf{s}\mf{l}(2,\R)\oplusrhrim \R^2)\oplus\R$} In the null cone as shown in \cite{nullcone}. Indeed, due to the fact that the $\mf{s}\mf{l}(2,\R)\oplusrhrim \R^2$ is in the $O(2,3)$-action, this 6-dimensional Lie algebra will be both in the null cone of the $O(2,4)$-action and the $O(3,3)$-action. 

 \item{$\mf{s}\mf{u}(2)\oplus \mf{s}_{3,\bullet}$:}These do not seem to be in the null cone, not even for $\mf{n}_{3,1}$. 
\item{$\mf{s}\mf{l}(2,\R)\oplus \mf{s}_{3,\bullet}$:} The case $\mf{s}\mf{l}(2,\R)\oplus \mf{n}_{3,1}$ was shown to be in the null cone in \cite{nullcone}. For the proper solvable ones, $\mf{s}\mf{l}(2,\R)\oplus \mf{s}_{3,1-3}$, we can choose the following $\mf{s}\mf{l}(2,\R)=\span\{1,5,6\}$, and let $\mf{s}_{3,1-3}=\span\{2,3,4\}$, so that the possible non-zero structure constants are: 
\[ C^1_{~51}, ~C^6_{~56}, ~C^5_{~16}, \quad C^a_{~3b}, ~a,b=2,4.\]
Under the unusual class $[3,2,1]$, these Lie algebras are all in the null cone. 
\end{enumerate} 

\subsection{Non-trivial Levi decomposable Lie algebras of dimension 6}
There are also some non-trivial Levi-decomposable Lie algebras. They are listed in Table \ref{NTLD6} and all were found to be in the null cone in earlier works.   
\begin{table}[ht]
    \centering
    \begin{tabular}{|c||c|c|}
    \hline 
 Lie algebra & null cone & Comment\\
 \hline\hline 
$\mf{s}\mf{u}(2)\oplusrhrim \R^3$& $\in$ & Shown in \cite{Markestad}, of class $[1,1,1]$.\\
\hline 
$\mf{s}\mf{l}(2,\R)\oplusrhrim \R^3$& $\in$ & Shown in \cite{Markestad}, of class $[1,1,1]$\\
\hline 
$\mf{s}\mf{l}(2,\R)\oplusrhrim \mf{n}_{3,1}$& $\in$ & Shown in \cite{nullcone}, of class $[2,1]$ \\
\hline
$\mf{s}\mf{l}(2,\R)\oplusrhrim \mf{s}_{3,1, a=1}$& $\in$ & Shown in \cite{nullcone}, of class $[3,1,1]$ \\
\hline 
\end{tabular}
\caption{A list of non-trivial Levi decomposable Lie algebras of dimension 6.}\label{NTLD6}
\end{table}

\subsection{Semi-simple 6-dimensional Lie algebras}
None of these are in the null cone. The semi-simple Lie algebras of dimension 6 are $\mf{s}\mf{l}(2,\R)\oplus\mf{s}\mf{l}(2,\R)$, $\mf{s}\mf{u}(2)\oplus\mf{s}\mf{l}(2,\R)$, $\mf{s}\mf{u}(2)\oplus\mf{s}\mf{u}(2)$ and $\mf{s}\mf{l}(2,\C)$ (as a real Lie algebra). 

\section{Lie algebras with semi-simple subalgebras}

Markestad showed that for a semi-simple Lie algebra $\mf{g}$ the Lie algebras of type $\mf{g}\oplus\R^{p+k}$, where $\dim(\mf{g})=p$ is in the null cone of the $O(p,p+k)$ action. This can be seen by choosing $N_-=\mf{g}$. Furthermore, since $N_+\doplus H$, has to be nilpotent, then the Iwasawa decomposition will indicate a lower bound on the total dimension \cite{Knapp}: 
\[ \mf g = \mf k \doplus \mf a \doplus \mf n, \]
where $\mf k$ is compact, $\mf a$ is the maximal abelian Cartan subalgebra, and $\mf n$ is nilpotent. 

Here, we will investigate Lie algebras $\mf{g}\oplus\R^m$ where $\mf{g}$ is semi-simple and try to find examples with smallest possible dimension of $\R^{m}$. In all the examples given $k=0$ so we are considering the neutral case. In this pursuit  the restricted-root decomposition will also be useful \cite{Knapp}: 
\[ {\mf g}={\mf g}_0\doplus \underset{\lambda\in\Sigma}{ \dot\bigoplus \hspace{0.1cm}}{\mf g}_\lambda,\] 
where $\Sigma$ is the set of roots. In our case, we note that the nilpotent subalgebra ${\mf n}=\dot\bigoplus_{\lambda\in\Sigma^+}{\mf g}_\lambda$ given by the set of positive roots will form the basis for a subspace of  $N_+$. Now, in the following, we will use the courser decomposition defined by the lower central series: 
\[ {\mf n}_1:={\mf n}, \quad {\mf n}_{\lambda+1}:=[{\mf n},{\mf n}_\lambda],  \]
where ${\mf g}_\lambda:={\mf n}_{\lambda}/{\mf n}_{\lambda+1}$. Based on the restricted-root space decomposition, we can in general use the $Z$-grading of the Lie algebra: 
\[ {\mf g}= \dotbigoplus{\lambda=-\Delta}{\Delta}{\mf g}_\lambda,\] 
where the negative-valued $\lambda$'s correspond to the negative roots, and the ${\mf g}_\lambda$'s are the  gradation 
of the Lie algebra. 

We will now study Lie algebras with semi-simple subalgebras of the form $\mf{g}\oplus\R^{m}$. The first examples will be found by using a 'brute force' method where we check each root and its commutator relations. In this case the root space decomposition is useful and we give  examples in lower dimensions. Then, when proceeding to the exceptional Lie algebra ${\mf f}_4$ we provide with a method which can be generalised to all split Lie algebras, and then show how it can be further generalised to \emph{all} semi-simple Lie algebras. 

\subsection{Split $A_n$, $SL(n+1,\R)$}
For split $A_1$ and $A_2$ examples were given in \cite{nullcone}. Here it was shown that $\mf{s}\mf{l}(2,\R)\oplus \R$ and  $\mf{s}\mf{l}(3,\R)\oplus \R^2$ of dimension 4 and 10 respectively, are in the null cone of the $O(2,2)$ and  $O(5,5)$. When studying the split semi-simple algebras the Hasse diagrams prove to be very useful. For $A_n$ they are particularly simple, and using the \LaTeX-package 'lie-hasse' they can easily been displayed. For example, for $A_3$, $A_4$ and $A_5$ they are: \\
$~~A_3:\quad $\begin{dynkinDiagram}[edge length=1cm]{A}{3}
\hasse{A}{3}
\end{dynkinDiagram},
$~~A_4:\quad $\begin{dynkinDiagram}[edge length=1cm]{A}{4}
\hasse{A}{4}
\end{dynkinDiagram},
$~~A_5:\quad $\begin{dynkinDiagram}[edge length=1cm]{A}{5}
\hasse{A}{5}
\end{dynkinDiagram}\\

These diagrams illustrate the positive root diagrams of the simple Lie algebras. These diagrams themselves give directly the nilpotent subalgebra from the Iwasawa-decomposition for the split groups (the diagrams for the complex Lie algebras and the split are the same). For simplicity, we will denote a positive root by a sequence of numbers, e.g., $(0,1,1,1)=0\alpha_1+1\alpha_2+1\alpha_3+1\alpha_4$, where $\alpha_i$ are the simple roots. There is also a set of negative roots, which will be denoted $(0,-1,-1,-1)$, etc. We will also denote $H_i$ as the elements of the Cartan subalgebra. 

The following tables give the cases $\mf{s}\mf{l}(3,\R)\oplus \R^2$, $\mf{s}\mf{l}(4,\R)\oplus \R^3$ and $\mf{s}\mf{l}(5,\R)\oplus \R^4$ showing they are all in the null cone. Since these are the split Lie algebras, their complexified versions can also be read off these tables. 

\subsubsection*{$\mf{s}\mf{l}(3,\R)\oplus \R^2$, dimension 10:}

\begin{tabular}{c||c|c|c|c|c|}
\hline
Class  &{\bf5}&{\bf3}&{\bf3}&{\bf1}&{\bf 1 }\\
 \hline
$\theta^\#$& 2& 4& 6& 8& 10 \\
$N_+$& $\mathbb{R}$  & $\mathbb{R}$ & $(1,1)$ & $(1,0)$ & $(0,1)$ \\
\hline 
$N_-$ & $(-1,-1)$ & $(0,-1)$ & $(-1,0)$ & $H_1$ & $H_2$  \\ 
$\theta^\#$& 1& 3& 5& 7& 9 \\
\hline
\end{tabular}\\
This table can also be used for the complex version $\mf{s}\mf{l}(3,\C)\oplus \C^2$, considered as a real Lie algebra of dimension 20. In this case all the dimension will be doubled and we will get an example of class $[5,5,3,3,3,3,1,1,1,1]$: for all the roots, $(x,y)$, add also $(ix,iy)$.

\subsubsection*{$\mf{s}\mf{l}(4,\R)\oplus \R^3$, dimension 18:}
\begin{tabular}{c||c|c|c|c|c|c|c|c|c|c|}
\hline
Class  &{\bf7}&{\bf5}&{\bf5}&{\bf3}&{\bf3}&{\bf3}&{\bf1}&{\bf 1 }&{\bf 1}\\
 \hline
$\theta^\#$& 2& 4& 6& 8& 10& 12& 14& 16 & 18 \\
$N_+$& $\mathbb{R}$  & $\mathbb{R}$& $\R$ & $(1,1,1)$ & $(0,1,1)$ & $(1,1,0)$ & $($1,0,0)$$ & $$(0,1,0)$$ & $$(0,0,1)$$ \\
\hline 
$N_-$ & $$(-1,-1,-1)$$ & $$(0,-1,-1)$$ & $$(-1,-1,0)$$ & $$(-1,0,0)$$  & $$(0,0,-1)$$ & $$(0,-1,0)$$ & $H_1$ & $H_2$ & $H_3$ \\ 
$\theta^\#$& 1& 3& 5& 7& 9& 11& 13& 15  & 17\\
\hline
\end{tabular}\\
Again, this table can also be used for the complex version $\mf{s}\mf{l}(4,\C)\oplus \C^3$, considered as a real Lie algebra of dimension 36 as explained above.  

\subsubsection*{$\mf{s}\mf{l}(5,\R)\oplus \R^4$, dimension 28:}
\begin{tabular}{c||c|c|c|c|c|c|c|c|}
\hline
Class  &{\bf11}&{\bf7}&{\bf7}&{\bf7}&{\bf7}&{\bf5}&{\bf5}\\
 \hline
$\theta^\#$& 2& 4& 6& 8& 10& 12& 14  \\
$N_+$& $\mathbb{R}$  & $\mathbb{R}$& $\R$ & $\R$ & (1,1,1,1) & (0,1,1,1) & (1,1,1,0)  \\
\hline 
$N_-$ & (-1,-1,-1,-1) & (-1,-1,-1,0) & (0,-1,-1,-1) & (-1,-1,0,0)  & (0,0,-1,-1) & (0,-1,-1,0) & (-1,0,0,0)    \\ 
$\theta^\#$& 1& 3& 5& 7& 9& 11& 13\\
\hline
\end{tabular}\\

$\qquad \cdots$\quad \begin{tabular}{|c|c|c|c|c|c|c|c|}
\hline
{\bf3}&{\bf3}&{\bf3}&{\bf1}&{\bf1}&{\bf 1 }&{\bf 1}\\
 \hline
 16& 18& 20& 22& 24& 26 & 28 \\
  (0,0,1,1) & (1,1,0,0) & (0,1,1,0) & (1,0,0,0) & (0,1,0,0)$$ & $$(0,0,1,0)$$ & $$(0,0,0,1)$$ \\
\hline 
   $$(0,0,0,-1)  & $$(0,0,-1,0)$$ & $$(0,-1,0,0)$$ & $H_1$ & $H_2$ & $H_3$ & $H_4$\\ 
  15& 17& 19& 21& 23& 25  & 27\\
\hline
\end{tabular}\\
Again, table can also be used for the complex version $\mf{s}\mf{l}(5,\C)\oplus \C^4$, considered as a real Lie algebra of dimension 56 as explained above.  

\subsection{Split $G_2 \times \mathbb{R}^2$}
Here we will consider the Lie algebra split $\mf{g}_2 \oplus \mathbb{R}^2$. It is useful to use the Hasse diagram of $\mf{g}_2$ showing the positive root system:\\
$G_2:\qquad $\begin{dynkinDiagram}[edge length=1cm]{G}{2}
\hasse{G}{2}
\end{dynkinDiagram}\\

The following table shows that we can find this Lie algebra in the null cone belonging to the 16-dimensional class $[15,11,9,7,5,3,1,1]$:\\
\begin{tabular}{c||c|c|c|c|c|c|c|c|c|}
\hline
Class  &{\bf15}&{\bf11}&{\bf9}&{\bf7}&{\bf5}&{\bf3}&{\bf1}&{\bf 1 }\\
 \hline
$\theta^\#$& 2& 4& 6& 8& 10& 12& 14& 16  \\
$N_+$& $\mathbb{R}$  & $\mathbb{R}$ & $(2,3)$ & $(1,3)$ & $(1,2)$ & $(1,1)$ & $(1,0)$ & (0,1) \\
\hline 
$N_-$ & $(-2,-3)$ & $(-1,-3)$ & $(-1,-2)$ & $(-1,-1)$  & $(-1,0)$ & $(0,-1)$ & $H_1$ & $H_2$  \\ 
$\theta^\#$& 1& 3& 5& 7& 9& 11& 13& 15  \\
\hline
\end{tabular}
\vspace{0.2cm}

\noindent We will later use a more general construction where we get another metric of a different class which is also in the null cone. 
\subsection{Split $F_4 \times \mathbb{R}^{4}$}
Here we will consider the Lie algebra split $\mf{f}_4 \oplus \mathbb{R}^4$. The Hasse diagram of $\mf{f}_4$ showing the positive root system is: \\
$F_4:\qquad $
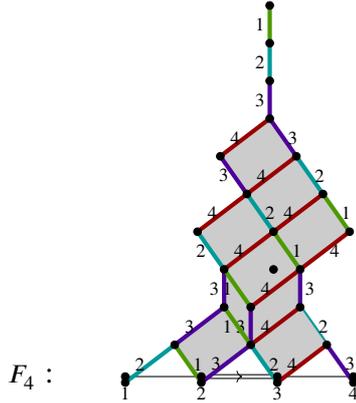
\begin{dynkinDiagram}[edge length=1cm]{F}{4}
\hasse{F}{4}
\end{dynkinDiagram}\\
Let us consider a construction for the split $\mf{f}_4 \oplus \mathbb{R}^4$ which can be generalised to any split simple Lie algebra. From the Hasse diagram we see that the positive roots can be decomposed as $\mf{g}_{1}\oplus\mf{g}_{2}\oplus\cdots\oplus\mf{g}_{11}$. The sequence of dimensions of these subspaces is 4, 3, 3, 3, 3, 2, 2, 1, 1, 1, 1. Similarly, the negative roots can be decomposed in the same way. Thus, we can give the root-space decomposition of $\mf{f}_4$, which again gives a $Z$-grading of the Lie algebra $\mf{f}_4$:
\[ \mf{f}_4=\mf{g}_{-11}\doplus\mf{g}_{-10}\doplus\cdots\doplus\mf{g}_{-1}\doplus\mf{a}\doplus\mf{g}_{1}\doplus\mf{g}_{2}\doplus\cdots\doplus\mf{g}_{11}=\dotbigoplus{\lambda=-11}{11}\mf{g}_\lambda,\]
where $\mf{g}_0=\mf{a}$ and $\doplus$ means direct sum as a vector space, but not (necessarily) Lie algebra direct sum. Indeed, we have 
\[ [\mf{g}_\lambda,\mf{g}_\sigma]\begin{cases} \subset\mf{g}_{\lambda+\sigma}, & |\lambda+\sigma|\leq 11.\\
=0, & |\lambda+\sigma|>11
\end{cases} \] 

We now consider the following scheme of pairing (if $\dim\mf{g}_\lambda<\dim\mf{g}_{-\lambda+1}=\dim \mf g_{\lambda-1}$ we $\oplus\R$ so that the dimensions match):\\
\begin{tabular}{c||c|c|c|c|c|c|c|c|}
\hline
Class  &{\bf 23} &{\bf 21} &$\cdots$ & $2\lambda+1$&$2\lambda-1$ & $\cdots$ &{\bf 3}&{\bf 1 }\\
 \hline
$N_+$& $\mathbb{R}$  & $\mf{g}_{11}$ & $\cdots$ & $\mf{g}_{\lambda+1}(\oplus\R)$ & $\mf{g}_{\lambda}(\oplus\R)$ & $\cdots$ & $\mf{g}_{2}\oplus\R$ & $\mf{g}_{1}$ \\
\hline 
$N_-$ & $\mf{g}_{-11}$ & $\mf{g}_{-10}$ & $\cdots$ & $\mf{g}_{-\lambda}$  & $\mf{g}_{-\lambda+1}$ & $\cdots$ & $\mf{g}_{-1}$ & $\mf{a}$  \\ 
\hline
\end{tabular}\\
Thus, $N_-=\dot\bigoplus_{\lambda=0}^{11}\mf{g}_{-\lambda}$, and $N_+=\R^4\oplus\dot\bigoplus_{\lambda=1}^{11}\mf{g}_{\lambda}$. 
We can now see that the following inequality is satisfied for all non-zero structure constants: 
\[ 23b_1+21b_2+\cdots (25-2i)b_i\cdots +b_{n}\leq -1\]
Indeed, by $[\mf{g}_\lambda,\mf{g}_\sigma] \subset\mf{g}_{\lambda+\sigma}$, the left hand side of the equation gives 
\[(2\lambda-1)+(2\sigma-1)-(2(\lambda+\sigma)-1)=-1,\] 
regardless of the sign of $\lambda$, or whether $\lambda<\sigma$ or not. Hence, the null cone requirement is indeed satisfied. This shows that Split $\mf{f}_4 \oplus \mathbb{R}^{4}$ is in the null cone of the $O(28,28)$-action. 

\subsection{Split $G\times\R^m$, $G$ simple}
Using the above construction, we can now show that: 
\begin{thm}
    Given one of the simple split Lie algebras $\mf{g}$. Assume that the Cartan subalgebra $\mf{a}$ is of dimension $m$. Then the Lie algebra $\mf{g}\oplus \R^m$ is in the null cone of the $O(p,p)$ action where $\dim(\mf{g}\oplus \R^m) = 2p$. 
\end{thm}
\begin{proof}
The proof is constructive and follow the lines of the Split $F_4\times \R^4$ case. In general, we use $Z$-grading of the split $\mf{g}$ using the  a root space decomposition, as explained above: 
\[ \mf{g}=\dotbigoplus{\lambda=-\Delta}{\Delta}\mf{g}_{\lambda},\]
where $\mf{g}_0=\mf{a}$ and 
\[ [\mf{g}_\lambda,\mf{g}_\sigma]\begin{cases} \subset\mf{g}_{\lambda+\sigma}, & |\lambda+\sigma|\leq \Delta.\\
=0, & |\lambda+\sigma|>\Delta. 
\end{cases} \] 
The dimensions of $\mf{g}_{\lambda}$ are different, but we add appropriate abelian dimensions $\R$, as follows: For each positive $\lambda$:
\[ \tilde{\mf{g}}_\lambda:=\mf{g}_{\lambda}\oplus \R^{m_{\lambda}}, \qquad m_{\lambda}=\dim\mf{g}_{\lambda-1}-\dim\mf{g}_{\lambda}, \quad \lambda=1,2,...,\Delta.\]
Notice that 
\[ \dim\mf{g}_{\Delta}+\sum_{\lambda=1}^{\Delta}m_\lambda=\dim\mf{g}_0=\dim\mf{a}=m\]
due to $\dim\mf{g}_{-\lambda}=\dim\mf{g}_{\lambda}$. We do a  similar pairing: \\
\begin{tabular}{c||c|c|c|c|c|c|c|c|}
\hline
Class  &${ 2\Delta+1}$ & ${ 2\Delta-1}$&$\cdots$ & $2\lambda+1$&$2\lambda-1$ & $\cdots$ &{\bf 3}&{\bf 1 }\\
 \hline
$N_+$& $\mathbb{R}^{\dim\mf{g}_{\Delta}}$  & $\tilde{\mf{g}}_{\Delta}$ & $\cdots$ & $\tilde{\mf{g}}_{\lambda+1}$ & $\tilde{\mf{g}}_{\lambda}$ & $\cdots$ & $\tilde{\mf{g}}_{2}$ & $\tilde{\mf{g}}_{1}$ \\
\hline 
$N_-$ & $\mf{g}_{-\Delta}$ & $\mf{g}_{-\Delta+1}$ & $\cdots$ & $\mf{g}_{-\lambda}$  & $\mf{g}_{-\lambda+1}$ & $\cdots$ & $\mf{g}_{-1}$ & $\mf{a}$  \\ 
\hline
\end{tabular}\\
Thus, $N_-=\dot\bigoplus_{\lambda=0}^{\Delta}\mf{g}_{-\lambda}$, and $N_+=\mathbb{R}^{\dim\mf{g}_{\Delta}}\oplus\dot\bigoplus_{\lambda=1}^{\Delta}\tilde{\mf{g}}_{\lambda}=\R^m\oplus\dot\bigoplus_{\lambda=1}^{\Delta}\mf{g}_{\lambda}$. 
For all non-zero structure constants from $[\mf{g}_\lambda,\mf{g}_\sigma] \subset\mf{g}_{\lambda+\sigma}$, we again get 
\[(2\lambda-1)+(2\sigma-1)-(2(\lambda+\sigma)-1)=-1,\] 
and the Lie algebra $\mf{g}\oplus \R^m$ is thus in the null cone. 
\end{proof}
\paragraph{Example: Split $SO(q,q)$.}
The split $SO(q,q)$ groups are of real rank $q(=\dim\mf{a})$ and the nilpotent subalgebra of positive roots, $\mf{n}=\dot\bigoplus_{\lambda=1}^\Delta\mf{g}_\lambda$ is of dimension $q(q-1)$. By the above construction, the Lie algebra $\mf{o}(q,q)\oplus\R^q$ is in the null cone of the $O(q^2,q^2)$-action. The Hasse diagram of, for example, the Lie algebras $D_4$ and $D_5$ are:\\
$D_4:\qquad $\begin{dynkinDiagram}[edge length=1cm]{D}{4}
\hasse{D}{4}
\end{dynkinDiagram}$\quad$ 
$D_5:\qquad $\begin{dynkinDiagram}[edge length=1cm]{D}{5}
\hasse{D}{5}
\end{dynkinDiagram}\\
These represent the split $SO(4,4)$ and $SO(5,5)$ groups. The dimensions of their lower central series are $[12, 8,5,2,1,0]$ and $[20, 15, 11, 7,4,2,1,0]$, respectively. This implies that the dimensions of the $\mf{g}_\lambda$ subspaces are $[4,3,3,1,1]$ and $[5,4,4,3,2,1,1]$. Hence the class of each of these are: 
\begin{itemize}
    \item Split $\mf{s}\mf{o}(4,4)\oplus\R^4$: Of dimension 32, in the null cone of the $O(16,16)$ action, and in the class $[11, 9,7_{\times3},5_{\times 3},3_{\times 4},1_{\times4}]$
    \item Split $\mf{s}\mf{o}(5,5)\oplus\R^5$: Of dimension 50, in the null cone of the $O(25,25)$ action, and in the class $[15,13, 11_{\times2}, 9_{\times3},7_{\times4},5_{\times 4},3_{\times 5},1_{\times5}]$
\end{itemize}
\paragraph{Example: The split exceptional groups.} The split groups $G_2$ and $F_4$ were discussed earlier. For the split $E_6$, $E_7$ and $E_8$ groups the real rank is 6, 7, and 8, respectively. By the above construction, the following Lie algebras are in the null cone: 
\begin{itemize}
    \item Split $\mf{e}_6\oplus\R^6$: Of dimension 84, and in the null cone of the $O(42,42)$-action. 
    \item Split $\mf{e}_7\oplus\R^7$: Of dimension 140, and in the null cone of the $O(70,70)$-action. 
    \item Split $\mf{e}_8\oplus\R^8$: Of dimension 256, and in the null cone of the $O(128,128)$-action. 
\end{itemize}
Table \ref{SplitLCS} gives the dimension of the split exceptional groups with the dimensions of the lower central series (LCS) and the decending series of the solvable Lie algebra $\mf{a}\oplus\mf{n}$. All of these can be read of the corresponding Hasse diagrams. 
\begin{table}[ht]
    \centering
\begin{tabular}{|c||c|c|c|c|c|}
 \hline\hline 
 $G$ & $\dim(\mf{a})$ & $\dim(\mathfrak{n})$ & LCS & $\dim(G)$ \\
\hline \hline 
Split $G_2$  & $2$ & $6$ & $[6,4, 3,2, 1, 0]$  & $14$ \\
 \hline 
Split $F_4$  & $4$ & $24$ & $[24,20,17,14,11,8,6,4,3,2,1, 0]$  & $52$ \\
 \hline 
 Split $E_6$ &   $6$&  $36$ &   $[36,30,25,20,15,11,8,5,3,2,1, 0]$  & $78$ \\
 \hline 
  Split $E_7$ &  $7$&  $63$ &   $\begin{matrix} [63,56,50,44,38,32,27,22,~~~~\\~~~~~~18,14,11,8,6,4,3,2,1, 0]\end{matrix} $  & $133$ \\
 \hline 
  Split $E_8$ &   $8$&  $120$ &   $\begin{matrix}[120,112,105,98,91,84,77, 70,~~~~ \\
 ~~64,58,52,46,41,36,32,28,24,20, \\ ~~~\quad 17,14,12,10,8,6,5,4,3,2,1, 0]\end{matrix}$  & $248$ \\
 \hline 
\end{tabular}
\caption{A table of the split exceptional Lie algebras with the dimension of the Cartan subalgebra and the nilpotent subalgebra of the positive roots. Also the dimension of the lower central series (LRS) and the dimension of the group itself.}
\label{SplitLCS}
\end{table}

Table \ref{Splitnullcone} gives the dimension of the grading $\mf{g}_\lambda>0$ and the corresponding class of the Lie algebras in the null cone. We note that the class of split $\mf{g}_2\oplus\R^2$ using this general construction is different from the one earlier found by "brute force". This indicates that there may be more non-isometric classes of these exceptional Lie algebras in the null cone. 
\begin{table}[ht]
    \centering
\begin{tabular}{|c||c|c|c|c|c|}
 \hline\hline 
 $\mf{g}$ & $\dim(\mf{g}_{\lambda>0})$  & Class \\
\hline \hline 
Split $\mf{g}_2\oplus\R^2$   & $[2,1,1,1,1]$  & $[11,9,7,5,3,3,1,1]$ \\
 \hline 
Split $\mf{f}_4\oplus\R^4$  &  $[4,3,3,3,3,2,2,1,1,1,1]$  & $\begin{matrix}[23,21,19,17,15_{\times 2},13_{\times 2},\qquad \\ \qquad11_{\times 3},9_{\times 3},7_{\times 3},5_{\times 3},3_{\times 4},1_{\times 4}]\end{matrix}$ \\
 \hline 
 Split $\mf{e}_6\oplus\R^6$ &    $[6,5,5,5,4,3,3,2,1,1,1]$  & $\begin{matrix} [23,21,19,17_{\times 2},15_{\times3},13_{\times 3},\qquad \\\qquad 11_{\times 4},9_{\times 5},7_{\times 5},5_{\times 5},3_{\times 6},1_{\times 6}]\end{matrix}$ \\
 \hline 
  Split $\mf{e}_7\oplus\R^7$ &    $\begin{matrix} [7,6,6,6,6,5,5,4,~~~~\\~~~~~~4,3,3,2,2,1,1,1,1]\end{matrix} $  & $\begin{matrix} [35,33,31,29,27_{\times2},25_{\times2}, 23_{\times 3}, 21_{\times3},19_{\times 4}, \quad \\ \quad17_{\times 4},  15_{\times 5},13_{\times 5}, 11_{\times 6},9_{\times 6},7_{\times 6},5_{\times 6},3_{\times 7},1_{\times7}]\end{matrix}$ \\
 \hline 
  Split $\mf{e}_8\oplus\R^8$ &     $\begin{matrix}[8,7,7,7,7,7,7, 6,6,6,~~~~ \\
 ~~6,5,5,4,4,4,4,3,3, \\ ~~~\quad 2,2,2,2,1,1,1,1,1,1]\end{matrix}$  & $\begin{matrix} [59,57,55,53,51,49,47_{\times2},45_{\times2},43_{\times2},  41_{\times2}, \quad \\ \quad 39_{\times 3},37_{\times3}, 35_{\times 4},33_{\times 4},31_{\times 4},29_{\times 4},\quad \\ \quad 27_{\times5},25_{\times5},  23_{\times 6}, 21_{\times6},19_{\times 6}, 17_{\times 6},  \quad \\ \qquad 15_{\times 7},13_{\times 7}, 11_{\times 7} ,9_{\times 7},7_{\times 7},5_{\times 7},3_{\times 8},1_{\times8}]\end{matrix}$ \\
 \hline 
\end{tabular}
\caption{The null cone Lie algebras constructed from the exceptional split Lie algebras. The list of dimensions of the subspaces $\mf{g}_{\lambda>0}$ along with their classes are also given. For the complex versions, $\mf{g}^\C_2$, $\mf{f}^\C_4$, $\mf{e}^\C_{6}$, $\mf{e}^\C_{7}$, and $\mf{e}^\C_{8}$   considered as a real Lie algebras, we can use the same table as for the split Lie algebras, but all the dimensions of the spaces double. } \label{Splitnullcone}
\end{table}

\subsection{General $G\times\R^m$, $G$ real and semi-simple}
For a real semi-simple group $G$ we can do a similar restricted root space decomposition to get a $Z$-grading of $\mf{g}$: 
\[ \mf{g}=\dotbigoplus{\lambda=-\Delta}{\Delta}\mf{g}_{\lambda},\]
but where $\mf{g}_0=\mf{m}_0+\mf{a}$ is not necessarily abelian, yet still:  
\[ [\mf{g}_\lambda,\mf{g}_\sigma]\begin{cases} \subset\mf{g}_{\lambda+\sigma}, & |\lambda+\sigma|\leq \Delta.\\
=0, & |\lambda+\sigma|>\Delta. 
\end{cases}. \] 
Again the dimensions of $\mf{g}_{\lambda}$ are different, but we add an appropriate abelian dimension $\R$, for each positive $\lambda$:
\[ \tilde{\mf{g}}_\lambda:=\mf{g}_{\lambda}\oplus \R^{m_{\lambda}}, \qquad m_{\lambda}=\dim\mf{g}_{\lambda-1}-\dim\mf{g}_{\lambda}, \quad \lambda=1,2,...,\Delta.\]
Notice that in this case
\[ \dim\mf{g}_{\Delta}+\sum_{\lambda=1}^{\Delta}m_\lambda=\dim\mf{g}_0=\dim\mf{a}+\dim\mf{m}_0=m.\]
We do a  similar pairing: \\
\begin{tabular}{c||c|c|c|c|c|c|c|c|}
\hline
Class  &${ 2\Delta+1}$ & ${ 2\Delta-1}$&$\cdots$ & $2\lambda+1$&$2\lambda-1$ & $\cdots$ &{\bf 3}&{\bf 1 }\\
 \hline
$N_+$& $\mathbb{R}^{\dim\mf{g}_{\Delta}}$  & $\tilde{\mf{g}}_{\Delta}$ & $\cdots$ & $\tilde{\mf{g}}_{\lambda+1}$ & $\tilde{\mf{g}}_{\lambda}$ & $\cdots$ & $\tilde{\mf{g}}_{2}$ & $\tilde{\mf{g}}_{1}$ \\
\hline 
$N_-$ & $\mf{g}_{-\Delta}$ & $\mf{g}_{-\Delta+1}$ & $\cdots$ & $\mf{g}_{-\lambda}$  & $\mf{g}_{-\lambda+1}$ & $\cdots$ & $\mf{g}_{-1}$ & $\mf{a}+\mf{m}_0$  \\ 
\hline
\end{tabular}\\
where  $N_-=\dot\bigoplus_{\lambda=0}^{\Delta}\mf{g}_{-\lambda}$, and $N_+=\mathbb{R}^{\dim\mf{g}_{\Delta}}\oplus\dot\bigoplus_{\lambda=1}^{\Delta}\tilde{\mf{g}}_{\lambda}=\R^m\oplus\dot\bigoplus_{\lambda=1}^{\Delta}\mf{g}_{\lambda}$. 

Thus we have proven that: 
\begin{thm}
    Given a semi-simple real Lie algebra $\mf{g}$. Assume that the subalgebra $\mf{a}\oplus \mf{m}_0$ in the restricted-root decomposition is of dimension $m$. Then the Lie algebra $\mf{g}\oplus \R^m$ is in the null cone of the $O(p,p)$-action where $p=\tfrac 12\dim(\mf{g}\oplus \R^m) $. 
\end{thm}

\paragraph{Example: Real form $SO(p+1,2)$, $B_n$ I and $D_n$ I.} For the real forms $SO(p+1,2)$, $p\geq 2$, we have the dimension of $\mf{m}_0$ equal to $(p-1)(p-2)/2$, and the real rank is 2$(=\dim\mf{a})$. The dimensions of the lower central series of $\mf{n}$ is $[2p,p,1,0]$. Consequently, the dimensions of $\mf{g}_{\lambda}$, $\lambda=1, 2, 3$ are $p$, $p-1$, 1. The above construction therefore gives us a $(p^2+p+6)$-dimensional pseudo-Riemannian metric in the null cone of the Lie algebra $\mf{s}\mf{o}(p+1,2)\oplus\R^m$, where $m=(p^2-3p+6)/2$. It is of class $[7,5_{\times (p-1)}, 3_{\times p},1_{\times m}]$.

\paragraph{Example: Real form $Sp(8,4)$, $C_6$ II.} For the real form $Sp(8,4)$ the dimension of $\mf{m}_0$ is 16 and is of real rank 2$(=\dim\mf{a})$. Hence, $m=18$. The dimension of the nilpotent subalgebra consisting of the positive roots $\mf{n}=\dot\bigoplus_{\lambda=1}^\Delta\mf{g}_\lambda$ is 30. Computing the dimensions of the lower central series of the nilpotent subalgebra $\mf{n}$ we get the sequence $[30,14,7,3,0]$. The dimensions of $\mf{g}_{\lambda}$, $\lambda=1, 2, 3, 4$ are therefore 16, 7, 4, 3, respectively. The above construction places the 96-dimensional Lie algebra of $Sp(8,4)\times \R^{18}$ in the null cone of the $O(48,48)$-action and is of  class $[9_{\times 3},7_{\times 4},5_{\times 7}, 3_{\times 16}, 1_{\times 18}]$. 


\paragraph{Example: Real form $F_4^{-20}$, FII.} The real form $F_4^{-20}$ of the exceptional Lie group $F_4$ has $\dim\mf{m}_0=21$, real rank 1, and the nilpotent subalgebra consisting of the positive roots $\mf{n}=\dot\bigoplus_{\lambda=1}^\Delta\mf{g}_\lambda$ has the following sequence of dimensions of the lower central series: $[15,7,0]$. Thus, the dimensions of $\mf{g}_{\lambda}$, $\lambda=1, 2$ are 8, and 7, respectively. Consequently, the 74-dimensional Lie algebra $\mf{f}^{-20}_4\oplus\R^{22}$ is in the null cone of the $O(37,37)$-action and is of class $[5_{\times 7},3_{\times 8}, 1_{\times 22}]$.  

\paragraph{Example: Real form $E_7^{-25}$, EVII.} The real form $E_7^{-25}$ of the exceptional Lie group $E_7$ has $\dim\mf{m}_0=28$, real rank 3, and the nilpotent subalgebra consisting of the positive roots $\mf{n}=\dot\bigoplus_{\lambda=1}^\Delta\mf{g}_\lambda$ has the following sequence of dimensions 17, 16, 9, 8, 1. Thus the 164-dimensional Lie algebra $\mf{e}_7^{-25}\oplus\R^{31}$ is in the null cone of the $O(82,82)$ action and is of class $[11,9_{\times 8},7_{\times 9},5_{\times 16},3_{\times 17},1_{\times 31}]$.

\paragraph{Example: Real form $E_8^{-24}$, EIX.} The real form $E_8^{-24}$ of the exceptional Lie group $E_8$ has $\dim\mf{m}_0=28$, real rank 4, and the nilpotent subalgebra consisting of the positive roots $\mf{n}=\dot\bigoplus_{\lambda=1}^\Delta\mf{g}_\lambda$ has the following sequence of dimensions 18,17,17,17, 10, 9, 9, 8,1,1,1. Thus the 280-dimensional Lie algebra $\mf{e}_8^{-24}\oplus\R^{32}$ is in the null cone of the $O(140,140)$-action and is of class\\ $[23,21,19,17_{\times 8},15_{\times9},13_{\times 9}, 11_{\times 10},9_{\times 17},7_{\times 17},5_{\times 17},3_{\times 18},1_{\times 32}]$.

\paragraph{Semi-simple algebras $\mf{g}=\oplus_i\mf{g}_i$, where $\mf{g}_i$ is simple.} In appendix \ref{ListSimple} all the real simple Lie algebras are listed with the dimensions of $\mf{m}_0$, $\mf{a}$ and total space $\mf{g}\oplus\R^m$. If the Lie algebra is sum of simple Lie algebras, hence, properly semi-simple, then we can do the direct sum of each level of grading. If for each $i$: 
\[ \mf{g}_i\oplus\R^{m_i}= \left(\dotbigoplus{\lambda=-\Delta_i}{0}\mf{g}_{i,\lambda}\right)\doplus \left(\dotbigoplus{\lambda=1}{\Delta_i}\tilde{\mf{g}}_{i, \lambda}\right),\]
we let $\Delta:=\max(\Delta_i)$ and define:
\[ {\mf g}_{\lambda}=\dotbigoplus{i}{}\mf{g}_{i,\lambda}, ~\lambda=-\Delta,..,0;  \qquad \tilde{\mf g}_{\lambda}=\dotbigoplus{i}{}\tilde{\mf{g}}_{i,\lambda},~\lambda=1,..,\Delta.\] 
This would provide with a construction of the algebra $\mf{g}\oplus\R^m$ placing it in the null cone for $\mf{g}=\oplus_i\mf{g}_i$  an arbitrary semi-simple Lie algebra, and $m=\sum_im_i=\sum_i\dim{\mf{g}_{i,0}}$.

\section*{Acknowledgements} 
The author would like to acknowledge a very useful discussion with Boris Kruglikov. 

\appendix

\section{Summary of spaces $\mf{g}\oplus\R^m$, $\mf{g}$ simple, in the null cone}
\label{ListSimple}
All of the cases here are in the null cone of the $O(p,p)$ action for $p=\tfrac 12\dim(\mf{g}\oplus\R^m)$. In the complex cases, the dimensions given are $\dim_\R(\mf{g}^\C_\lambda)$, where $\mf{g}_\lambda^\C$ are the complexified gradings of the real split case. 
\subsection{Classical algebras, split, compact and complex case, $\mf{g}\oplus\R^m$. }
\begin{tabular}{|c||c|c|c|c|c|}
 \hline\hline 
$\mf{g}$ & $\dim(\mf{m}_0)$ & $\dim(\mf{a})$ & $\dim(\mathfrak{n})$ & $m$ & $\dim(\mf{g}\oplus\R^m)$ \\
\hline \hline 
$\mf{s}\mf{l}(n+1,\R), n\geq 1$ &$0$ & $n$ & $\tfrac{n(n+1)}2$ &  $n$ & $n(n+3)$ \\
 \hline 
 $\mf{s}\mf{u}(n+1), n\geq 1$ &$n(n+2)$ & $0$ & $0$ &  $n(n+2)$ & $2n(n+2)$ \\
 \hline
 $\mf{s}\mf{l}(n+1,\C), n\geq 1$ &$0$ & $2n$ & ${n(n+1)}$ &  $2n$ & $2n(n+3)$ \\
 \hline\hline
 $\mf{s}\mf{o}(n+1,n), n \geq 1$ &$0$ & $n$ & $n^2$ &  $n$ & $2n(n+1)$ \\
 \hline 
 $\mf{s}\mf{o}(2n+1), n\geq 1$ &$n(2n+1)$ & $0$ & $0$ &  $n(2n+1)$ & $2n(2n+1)$ \\
 \hline
 $\mf{s}\mf{o}(2n+1,\C), n\geq 1$ &$0$ & $2n$ & ${2n^2}$ &  $2n$ & $4n(n+1)$ \\
 \hline\hline
 $\mf{s}\mf{p}(2n,\R), n \geq 2$ &$0$ & $n$ & $n^2$ &  $n$ & $2n(n+1)$ \\
 \hline 
 $\mf{s}\mf{p}(n), n\geq 2$ &$n(2n+1)$ & $0$ & $0$ &  $n(2n+1)$ & $2n(2n+1)$ \\
 \hline
 $\mf{s}\mf{p}(2n,\C), n\geq 2$ &$0$ & $2n$ & ${2n^2}$ &  $2n$ & $4n(n+1)$ \\
 \hline\hline
 $\mf{s}\mf{o}(n,n), n \geq 3$ &$0$ & $n$ & $n(n-1)$ &  $n$ & $2n^2$ \\
 \hline 
 $\mf{s}\mf{o}(2n), n\geq 3$ &$n(2n-1)$ & $0$ & $0$ &  $n(2n-1)$ & $2n(2n-1)$ \\
 \hline
 $\mf{s}\mf{o}(2n,\C), n\geq 3$ &$0$ & $2n$ & ${2n(n-1)}$ &  $2n$ & $4n^2$ \\
 \hline
\end{tabular}
\subsection{Classical algebras, other real forms, $\mf{g}\oplus\R^m$}
\begin{tabular}{|c||c|c|c|c|c|}
 \hline\hline 
$\mf{g}$ & $\dim(\mf{m}_0)$ & $\dim(\mf{a})$ & $\dim(\mathfrak{n})$ & $m$ & $\dim(\mf{g}\oplus\R^m)$ \\
\hline \hline 
 $\mf{s}\mf{u}^*(2n)$ &$3n$ & $n-1$ & $ 2n(n-1)$ & $ 4n-1$ & $4n^2+4n-2$   \\
 \hline 
$\mf{s}\mf{u}(p,q), p\geq q$ &$\begin{matrix}(p-q)^2 \quad \\ \quad +q-1\end{matrix}$ & $q$ & $q(2p-1)$ & $\begin{matrix}(p-q)^2\quad \\ \quad+2q-1\end{matrix}$  & $\begin{matrix}2(p^2+q^2)\quad \\ \quad+2q-2\end{matrix}$  \\
\hline
 $\mf{s}\mf{o}(p+1,q), p> q$ &$\begin{matrix}\tfrac12(p-q+1) \\ \quad\times (p-q)\end{matrix}$ & $q$ & $pq$ & $\begin{matrix}\tfrac12[(p-q)^2\quad \\ \quad+p+q]\end{matrix}$  & $\begin{matrix}p^2+q^2\quad \\ \quad+p+q\end{matrix}$ \\
\hline
  $ \mf{s}\mf{p}(2p,2q)$, $p\geq q$ &$\begin{matrix}2(p-q)^2\quad \\ \quad+p+2q\end{matrix} $ & $q$ & $(4p-1)q$ & $\begin{matrix}2(p-q)^2\quad \\ \quad+p+3q\end{matrix}$  & $\begin{matrix}4(p^2+q^2)\quad \\ \quad+4q+2p\end{matrix}$ \\
  \hline
  $\mf{s}\mf{o}^*(2n)$, $n$ even &$2n$ & $\tfrac 12n$ & $ \tfrac12n(2n-3)$ & $ \tfrac 52n$ & $2n(n+1)$   \\
 \hline 
 $\mf{s}\mf{o}^*(2n)$, $n$ odd &$2n-1$ & $\frac12(n-1)$ & $\begin{matrix}\tfrac12(n-1)\quad \\ \quad\times(2n-1)\end{matrix}$ & $ \tfrac 12(5n-3)$ & $2(n^2+n-1)$   \\
 \hline 
\end{tabular}

\subsection{The exceptional Lie algebras, $\mf{g}\oplus\R^m$.}
\begin{tabular}{|c||c|c|c|c|c|}
 \hline\hline 
 $\mf{g}$ & $\dim(\mf{m}_0)$ & $\dim(\mf{a})$ & $\dim(\mathfrak{n})$ & $m$ & $\dim(\mf{g}\oplus\R^m)$ \\
\hline \hline 
$\mf{g}_2$ split &$0$ & $2$ & $6$ & $2$  & $16$ \\
 \hline 
 $\mf{g}_2$ compact &$14$ & $0$ & $0$ & $14$  & $28$ \\
 \hline 
$\mf{g}^{\C}_2$ &$0$ & $4$ & $12$ & $4$  & $32$ \\
 \hline \hline
$\mf{f}_4$ split &$0$ & $4$ & $24$ & $4$  & $56$ \\
 \hline 
 $\mf{f}^{-20}_4$, FII &$21$ & $1$ & $15$ & $22$  & $74$ \\
 \hline 
 $\mf{f}_4$ compact &$52$ & $0$ & $0$ & $52$  & $104$ \\
 \hline 
 $\mf{f}^\C_4$  &$0$ & $8$ & $48$ & $8$  & $112$ \\
 \hline \hline
 $\mf{e}_6$ split & $0$&  $6$&  $36$ &   $6$  & $84$ \\
 \hline 
 $\mf{e}^2_6$, EII & $2$&  $4$&  $36$ &   $6$  & $84$ \\
 \hline 
 $\mf{e}^{-14}_6$, EIII& $16$&  $2$&  $30$ &   $18$  & $96$ \\
 \hline 
 $\mf{e}^{-26}_6$, EIV & $28$&  $2$&  $24$ &   $30$  & $108$ \\
 \hline 
 $\mf{e}_6$  compact& $78$&  $0$&  $0$ &   $78$  & $156$ \\
 \hline 
 $\mf{e}^\C_6$ & $0$&  $12$&  $72$ &   $12$  & $168$ \\
 \hline \hline
  $\mf{e}_7$ split  & $0$&  $7$&  $63$ &   $7$ & $140$ \\
 \hline 
 $\mf{e}^{-5}_7$,  EVI& $9$&  $4$&  $60$ &   $13$ & $146$ \\
 \hline 
 $\mf{e}^{-25}_7$,  EVII& $28$&  $3$&  $51$ &   $31$ & $164$ \\
 \hline 
 $\mf{e}_7$ compact  & $133$&  $0$&  $0$ &   $133$ & $266$ \\
 \hline
  $\mf{e}^\C_7$   & $0$&  $14$&  $126$ &   $14$ & $280$ \\
 \hline\hline
  $\mf{e}_8$ split & $0$&  $8$&  $120$ &   $8$ & $256$\\
 \hline 
  $\mf{e}^{-24}_8$, EIX & $28$&  $4$&  $108$ &   $32$ & $280$\\
 \hline 
 $\mf{e}_8$  compact& $248$&  $0$&  $0$ &   $248$ & $496$\\
 \hline 
  $\mf{e}^\C_8$ & $0$&  $16$&  $240$ &   $16$ & $512$\\
 \hline 

\end{tabular}

\bibliography{bibl}

\end{document}